\documentclass{article}
\usepackage{graphicx} 
\usepackage{amsmath,amsthm}
\usepackage{amsfonts}
\usepackage{mathtools}
\usepackage{ amssymb }
\usepackage{ stmaryrd }
\usepackage{tikz-cd}
\usepackage[shortlabels]{enumitem}
\usepackage{ upgreek }
\usepackage{yfonts}
\usepackage{dsfont}
\usepackage{url}
\usepackage{graphicx}
\usepackage{hyperref}

\usepackage{afterpage}
\usepackage{ amssymb }
\usepackage{yfonts}
\usepackage[colorinlistoftodos]{todonotes}
\usepackage{enumitem}
\setlist[enumerate]{label=(\arabic*)}
\usepackage[left=3.5cm,right=3cm,twoside]{geometry}

\theoremstyle{plain}
\newtheorem{theorem}{Theorem}[section]
\newtheorem{proposition}[theorem]{Proposition}
\newtheorem{corollary}[theorem]{Corollary}
\newtheorem{lemma}[theorem]{Lemma}

\theoremstyle{definition}
\newtheorem{definition}[theorem]{Definition}

\newtheorem{remark}[theorem]{Remark}

\newtheorem{example}[theorem]{Example}

\newtheorem{question}[theorem]{Question}

\newtheorem*{theorem*}{Theorem}
\newtheorem*{definition*}{Definition}

\newcommand{\N}{\mathbb{N}}
\newcommand{\Z}{\mathbb{Z}}
\newcommand{\Q}{\mathbb{Q}}

\newcommand{\C}{\mathbb{C}}

\usepackage{titling}
\usepackage[affil-it]{authblk}
\usepackage{orcidlink}
\usepackage{todonotes}
\newcommand{\email}[1]{\href{mailto:#1}{\texttt{#1}}}

\makeatletter
\newcommand{\subjclass}[2][2020]{%
  \let\@oldtitle\@title%
  \gdef\@title{\@oldtitle\footnotetext{\hspace*{-2em}#1 \emph{Mathematics subject classification.} #2}}%
}
\newcommand{\keywords}[1]{%
  \let\@@oldtitle\@title%
  \gdef\@title{\@@oldtitle\footnotetext{\hspace*{-2em}\emph{Keywords:} #1.}}%
}

\makeatother

\author[ ]{Leonardo Raffaello Maximilian Gasparro%
  \thanks{email: \email{leonardo.r.m.gasparro@gmail.com}}}

\author[1]{Lorenzo Luperi Baglini\orcidlink {0000-0002-0559-0770}%
  \thanks{email: \email{lorenzo.luperi@unimi.it}}}

\affil[1]{\small Dipartimento di Matematica, Università  di Milano, Via Saldini 50, 20133 Milano, Italy}

\title{Weak, strong and mixed extensions of relations to spaces of ultrafilters}



\begin{document}
\maketitle

\begin{abstract}
The use of nonstandard methods to characterize properties of weak, strong and mixed extensions of congruences to ultrafilters has been the main topic of several recent papers, focused mostly on congruences and divisions. We show that similar methods can be used to extend these characterizations to arbitrary relations and their interplay.
\end{abstract}

\section{Introduction}

The problem of extending $n$-ary relations defined on a set $X$ to $n$-ary relations on the set $\beta X$ of ultrafilters on $X$ has often been considered in the literature, for example in the context of ultrafilter extensions of arbitrary first-order models, see e.g. \cite{goranko2007filter}, \cite{saveliev2011ultrafilter}. Whilst the case $n=1$ is trivial, for $n\geq 2$ there is, in general, no canonical procedure to perform such an extension. At the best of our knowledge, two are the extensions usually considered in the literature, that we will dub\footnote{These notions have appeared in the literature with a variety of different names and notations; we will motivate in Section \ref{sec:weakandstrong} our name choice.} weak and strong extension, and denote by $R^{w},R^{s}$, respectively. The weak extension was considered implicitly first as early as in 1951, in a paper by Jónsson
and Tarski (see \cite{JT1, JT2}), whilst the strong extension first appeared, at the best of our knowledge, in works of Saveliev, see \cite{SavelievTwoTypes,SavOrd}. 
A systematic comparative study of both extensions in the case of binary relations has been done in \cite{SavelievTwoTypes}. For a detailed and through study of these and other ultrafilters extensions, including several historical remarks, we refer to the seminal paper \cite{poliakov2021ultrafilter} by Poliakov and Saveliev.

Recently, weak, strong and other extensions of certain arithmetic relations have been studied, using methods coming from nonstandard analysis, with the dual goal of better understanding the algebraic properties of $\beta\N$ and of finding applications in arithmetic Ramsey theory. In particular, extensions of the binary relation $D(a,b):=$ ``$a$ divides $b$'' and of the ternary relation $C(a,b,c):=$ ``$b$ is congruent to $c$ modulo $a$'' have been intensively studied, in a series of papers, by S\v{o}bot; see e.g. \cite{vsobot2015divisibility,vsobot2021congruence}, or \cite{vsobot2024survey} for a survey of his main results. The extensions of $D(a,b)$ produce rich order structures on $\beta\N$, whose properties are not yet completely understood. For a related study of properties of the weak extension of orders to types\footnote{Extensions to spaces of ultrafilters can be seen as particular instances of these more general constructions.}, we refer to \cite{baglini2025extending}.

For our purposes, the extensions of $C(a,b,c)$ are more relevant. In \cite{vsobot2021congruence} S\v{o}bot proved that $C^{s}(p,u,v)$ is, for $p$ fixed, an equivalence relation but that, in general, it has also certain counterintuitive arithmetic properties; for example, some ultrafilters are not congruent to $0$ modulo themselves. This issue is avoided by a third kind of extension of the congruence, that S\v{o}bot in \cite{vsobot2021congruence} called  ``congruence'' and denoted by $\equiv_{p}$. This extension is, in general, different from both the strong and the weak extension; in the same paper, S\v{o}bot left as an open question wheter $\equiv_{p}$ itself is always an equivalence relation. This question was answered negatively in \cite{di2023self} by Di Nasso, Luperi Baglini, Mennuni, Pierobon and Ragosta. They proved that $\equiv_{p}$ itself is an equivalence relation solely when it coincides with $C^{s}(p,\cdot,\cdot)$, and that this happens precisely when $p$ has a fundamental property, expressible in terms of divisibility, called self-divisibility (see \cite[Definition 4.1]{di2023self}). Self-divisibility is an interesting combinatorial notion that admits a large amount\footnote{Sixteen, to be precise.} of equivalent characterizations, as shown in \cite[Theorem 7.7]{di2023self}. This notion has already been used to prove an extension of Hindman's Theorem involving both sums and divisions in \cite{SumQuot}. This has shown that notions arising from the extension of relations to $\beta\N$ can find nontrivial applications to combinatorial problems in Ramsey theory. The full potential of this approach has yet to be completely understood.

In this paper, our goal is to show that the main known results about the extensions of divisibility and congruence can be seen as particular instances of broader phenomenons, and to generalize the notion of self-divisibility. First, we consider the problem of which properties of $R$ are preserved by $R^{w},R^{s}$ in general; whilst the weak case is easily settled, as shown in Theorem \ref{widetilde-R-equivalence-relation-theorem}, the situation is usually harder in the strong case. Our main characterizing result for $R^{s}$, that generalizes the analogous result proven by S\v{o}bot for congruences in \cite{vsobot2021congruence}, is the following one.

\begin{theorem*}[A]
For an equivalence relation $R \subseteq I^2$, the following facts are equivalent.
\begin{enumerate}[label=(\roman*)]
        	\item $R^s$ is an equivalence;
	\item $R^s$ is reflexive;
    \item $R^s$ is symmetric and total;
            \item $R^{s}$ is total.
\end{enumerate}
\end{theorem*}

We then move to the problem of extending the notions of $\equiv_{p}$ and self-divisibility to arbitrary relations. The former is attained by the so-called mixed extensions $R^{m}$, the latter by the notion of self-$R$ ultrafilter. What complicates the relationships between $R^{m},R^{s}$ and self-$R$ ultrafilters is that, in the congruence-division case, the connection is due to certain specific algebraic properties of congruences, in particular the facts that a number is congruent to $0$ modulo itself, and that being congruent modulo some $n\in\N$ forces to be congruent modulo all divisors of $n$. We introduce the analogues of these algebraic properties for arbitrary relations, and name them $e$-reflexivity and $e$-transitivity; many relations that arise naturally when working in algebraic contexts satisfy them, as shown in Section \ref{generalization-of-self-divisible-ultrafilters-section}. Our principal result, which is a generalization of the main equivalences of \cite[Theorem 4.10]{di2023self}, settles the relationships between $R^{m},R^{s}$ and self-$R$ ultrafilters in the $e$-transitive and $e$-reflexive cases.

\begin{theorem*}[B] Let $I$ be a set, let $e\in I$ and let $R\subseteq (I\setminus\{e\})\times I^{2}$ be both $e$-transitive and $e$-reflexive. The following facts are equivalent:
\begin{enumerate}[label=(\roman*)]
\item $R^{m}(p,\cdot,\cdot)=R^{s}(p,\cdot,\cdot)$;
\item $p$ is a self-$R(\cdot,\cdot,e)$ ultrafilter.
\end{enumerate}
\end{theorem*}

The paper is organized as follows. Similarly to \cite{di2023self,vsobot2015divisibility,vsobot2021congruence}, our approach is based on a nonstandard characterization of ultrafilters, whose basic properties are recalled in Section \ref{Sec:nonstandard}. Section \ref{sec:weakandstrong} is divided in two subsections. In the first we give the nonstandard characterizations of the strong and weak extensions of arbitrary relations, in the second we show several results regarding the preservation (or not) of properties by these extensions, including Theorem A. In Section \ref{generalization-of-self-divisible-ultrafilters-section} we introduce the notion of self-$R$ ultrafilters, of mixed extensions of relations, we prove Theorem B and we provide some examples of its possible applications. Finally, in Section \ref{sec open questions} we pose some open questions that we deem to be relevant.

\paragraph{Funding}

 LLB is supported by the project PRIN 2022 ``Logical methods in combinatorics'', 2022BXH4R5, Italian Ministry of University and Research (MUR).

\section{Nonstandard preliminaries}\label{Sec:nonstandard}

In this paper we assume familiarity with the basics of nonstandard methods, for which we refer the reader to, e.g., \cite{goldblatt2012lectures}. All nonstandard extensions we use will be assumed to be saturated enough. As our approach will be based on viewing ultrafilters as points in nonstandard extensions, and many of our results will involve tensor products, we briefly recall some well-known facts we will need.

Let $X$ be an infinite set. Ultrafilters on $X$ and points of $^{\ast}X$ can be identified as follows: to any point $\alpha\in\,^{\ast}X$ we can associate the ultrafilter $u_{\alpha}:=\{ A \subseteq X \mid \alpha \in {^*A} \}$, whilst any ultrafilter $u\in\beta X$ can be identified with
$\mu(u) = \{ \alpha \in\,^{\ast}X \mid u = u_\alpha \}$. When $\alpha,\beta$ generate the same ultrafilter we will write $\alpha\sim\beta$. $\mu(u)$ is called the monad of $u$, and any $\alpha \in \mu(u)$ is called a generator of $u$. Notice that, when $u$ is the principal ultrafilter on $x\in X$, then $\mu(u)=\{x\}$. Our saturation assumption entails that every monad is nonempty.

\begin{remark}\label{remark:quantifier-switch} An immediate consequence of the definition of monad of $u$, that we will often use, is that generators allow to switch quantifiers, in the following sense: given $A\subseteq X, u\in\beta X$, the following facts are easily seen to be equivalent:
\begin{enumerate}[label=(\roman*)]
\item $A\in u$;
\item $\exists \alpha\in\mu(u) \ \alpha\in\,^{\ast}A$;
\item $\forall \alpha\in\mu(u) \ \alpha\in\,^{\ast}A$.
\end{enumerate}
\end{remark}

We refer to \cite{di2015hypernatural, dinassoNonstandardMethodsRamsey2019, baglini2012hyperintegers} for a more detailed study of the association between ultrafilters and nonstandard points and its applications in combinatorics.

As it is well-known, given infinite sets $I,J$, the relationship between $\beta I \times \beta J$ and $\beta (I\times J)$ is not straightforward, as for any $u\in\beta I,v\in\beta J$ the product 
$$u\times v:=\{A\times B\mid A\in u, B\in v\}$$
is a filter but not, in general, an ultrafilter on $I\times J$. Among all the ultrafilters that extend $u\times v$, the tensor product $u\otimes v$ is the one that plays a special role in many applications, including ours. Let us recall its definition.

\begin{definition} Let $I,J$ be sets. Given $u\in\beta I,v\in\beta J$, the tensor product $u\otimes v$ is the ultrafilter on $I\times J$ defined as follows: for all $A\subseteq I\times J$, we let
\[A\in u\otimes v\Leftrightarrow \{i\in I\mid \{j\in J\mid (i,j)\in A\}\in v\}\in u.\]
Given $(\alpha,\beta)\in \,^{\ast}(I\times J)$, we say that $(\alpha,\beta)$ is a \emph{tensor pair} if $u_{(\alpha,\beta)}$ is a tensor product. Similarly for tensor triples, tensor $k$-uples, etc.
\end{definition}

In the following proposition we list the known facts regarding ultrafilters and their tensors that we will need. All of these, except the last one, were already known; in particular in \ref{T6}, \ref{T7} we take the chance to amend two hypotheses that were wrongfully stated\footnote{Luckily enough for the second author of this paper, who was also the author of \cite{luperi2019nonstandard}, only slightly wrongfully stated.} in \cite{luperi2019nonstandard}.

\begin{proposition}\label{proprietà tensori} Let $I,J,H$ be sets, let $u\in\beta I,v\in\beta J, p\in \beta H$, let $\alpha\in\,^{\ast}I,\beta\in\,^{\ast}J,\gamma\in\,^{\ast}H$. The following facts hold:
\begin{enumerate}[label=(\roman*)]
\item\label{T1} $u\otimes (v\otimes p)=(u\otimes v)\otimes p$.
\item\label{T2} Let $I=J=\N$. Then $(\alpha,\beta)$ is a tensor pair if and only if for all $f:\N\rightarrow \N$ we have $f(\beta)\in\N$ or $\alpha<f(\beta)$.
\item\label{T3} If $t\supseteq u_{\alpha}\times u_{\beta}$ then for all $(\eta,\sigma)\in \mu(t)$ we have $\eta\sim\alpha,\sigma\sim\beta$.
\item\label{T4} For all $i\in I, j\in J$ $(i,\beta)$ and $(\alpha,j)$ are tensor pairs.
\item\label{T5} If $(\alpha,\beta)$ is a tensor pair, then $u_{(\alpha,\beta)}=u_{\alpha}\otimes u_{\beta}$.
\item\label{T6} $(\alpha,\beta,\gamma)$ is a tensor triple if and only if $(\alpha,\beta),(\beta,\gamma)$ and $(\alpha,\gamma)$ are tensor pairs.
\item\label{T7} If $\beta\notin J$ and $(\alpha,\beta),(\beta,\gamma)$ are tensor pairs, then $(\alpha,\gamma)$ is a tensor pair.
\item\label{T8} Let $t\in\beta(I\times J)$ be an ultrafilter that extends $u_{\alpha}\times u_{\beta}$. Then there exist $\alpha^{\prime}\sim\alpha$, $\beta^{\prime}\sim \beta$ such that $(\alpha,\beta^{\prime})\sim (\alpha^{\prime},\beta)\in\mu(t)$.
\item\label{T9} Let $A,B$ be sets, $f:I\rightarrow A, g:J\rightarrow B$ functions. If $(\alpha,\beta)$ is a tensor pair, then so is $(\,^{\ast}f(\alpha),\,^{\ast}g(\beta))$.
\end{enumerate}
    
\end{proposition}

\begin{proof}
    (i) Associativity of $\otimes$ is a well-known property, whose proof is trivial.

    (ii) This is Puritz's Theorem, see \cite{puritz1972skies}.

    (iii)-(iv) These are immediate, see \cite[Proposition 4.6]{luperi2019nonstandard}.

    (v) This follows easily from (iii).

    (vi) This is an instance of \cite[Theorem 4.22]{luperi2019nonstandard}; we note explicitly that, with respect to the cited reference, in here we are additionally assuming $(\alpha,\gamma)$ to be a tensor pair. This assumption was wrongfully missed in \cite{luperi2019nonstandard}, and it is needed in the case where $\beta\in J$.

    (vii) This is \cite[Corollary 4.24]{luperi2019nonstandard}; again, we note explicitly that the assumption $\beta\notin J$ was wrongfully missed in the cited reference.

    (viii) We prove the existence of $\beta^{\prime}$, as the existence of $\alpha^{\prime}$ can be proven similarly.

    We observe that for all $A\in t$, as $t\supseteq u_{\alpha}\times u_{\beta}$ necessarily $\pi_{1}(A)=\{i\in I\mid \exists j\in J \ (i,j)\in A\}\in u_{\alpha}$. Hence, $\alpha\in\,^{\ast}\pi_{1}(A)$. Therefore
    $$\{\{\eta\in\,^{\ast}J\mid (\alpha,\eta)\in\,^{\ast}A\}\mid A\in t\}$$
    is a family of internal sets with the finite intersection property; by saturation, its intersection is non emtpy. Let $\beta^{\prime}$ be in this intersection. Then, by construction, $(\alpha,\beta^{\prime})\in \mu(t)$, hence $\beta^{\prime}\in \mu(u_{\beta})$ by \ref{T3}.
		
		(ix) This is a rephrasing of \cite[Theorem 4.15.(7)]{luperi2019nonstandard}.
\end{proof}

\section{Weak and strong extensions of relations}\label{sec:weakandstrong}

As stated in the introduction, the two most-studied extensions of relations to ultrafilters are the weak and the strong, which are defined as follows: let $I_{1},\dots,I_{n}$ be sets and consider a relation $R \subseteq I_1 \times \ldots \times I_n$.

\begin{definition}\label{relation-canonical-extension}
The weak extension of $R$ is the relation $R^{w} \subseteq \beta I_1 \times \ldots \times \beta I_n$ such that, for all $u_{1},\dots,u_{n}\in \beta I_1 \times \ldots \times \beta I_n$, $R^{w}(u_1 , \ldots , u_n)$ holds if and only if
$$(\forall A_1 \in u_1) \, \ldots \, (\forall A_n \in u_n) \, (\exists a_1 \in A_1) \, \ldots \, (\exists a_n \in A_n) \, R(a_1 , \ldots , a_n) .$$
\end{definition}

\begin{definition}\label{relation-strong-extension}
The strong extension of $R$ is the relation $R^s \subseteq \beta I_1 \times \ldots \times \beta I_n$ such that, for all $u_{1},\dots,u_{n}\in \beta I_1 \times \ldots \times \beta I_n$,  $R^s(u_1 , \ldots , u_n)$ holds if and only if
$$\{ a_1 \in I_1 \mid \ldots \{ a_n \in I_n \mid R(a_1 , \ldots , a_n) \} \in u_n \ldots \} \in u_1 .$$
\end{definition}

As we are interested in studying these extensions from the point of view of nonstandard analysis, first we are going to give their characterizations in terms of nonstandard generators of ultrafilters.

\subsection{Nonstandard characterizations}\label{generalization-original-section}

The nonstandard characterizations of weak and strong extensions of arbitrary relations are analogous, and generalize, those given by {\v{S}}obot in \cite{vsobot2015divisibility, vsobot2021congruence} for the divisibility and the congruence.

\begin{theorem}\label{canonical-relation-extension-generator-theorem-01}
Let $I_{1},\ldots,I_{n}$ be nonempty sets, and let $R\subseteq I_1 \times \ldots \times I_n$. For all $u_1 \in \beta I_1 , \ldots , u_n \in \beta I_n$ the following are equivalent.
\begin{enumerate}[label=(\roman*)]
        	\item $R^{w}\left(u_1 , \ldots, u_n\right)$;
            	\item $\exists \alpha_1 \in \mu(u_1) \ldots \exists \alpha_n \in \mu(u_n) \, {^{*}R}(\alpha_1 , \ldots , \alpha_n)$;
                \item $\exists p\supseteq u_1\times\ldots\times u_n$ such that $R\in p$;
                
                \item $\forall i\leq n$ $\forall \alpha_{i}\in \mu(u_{i})\ \forall j\leq n, j\neq i \ \exists \alpha_j \in \mu(u_j)\,^{\ast}R(\alpha_1 , \ldots , \alpha_n)$.

\end{enumerate}
\end{theorem}
\begin{proof}

(i) $\Rightarrow$ (ii) For all $A_1 \in u_1 , \ldots , A_n \in u_n$ let
$$\Phi_{A_1 , \ldots , A_n} := \{ (a_1 , \ldots , a_n) \in A_1 \times \ldots \times A_n \mid R(a_1 , \ldots , a_n) \} .$$
By Definition \ref{relation-canonical-extension}, $\Phi_{A_1 , \ldots , A_n} \not = \emptyset$. Clearly, the family $\{ \Phi_{A_1 , \ldots , A_n} \}_{A_1 \in u_1 , \ldots , A_n \in u_n}$ has the finite intersection property, hence by saturation $^{\ast}\Phi_{A_1 , \ldots , A_n} \not = \emptyset$. Any point in this intersection witnesses (ii).

(ii) $\Rightarrow$ (iii) Let $\alpha_{1},\ldots,\alpha_{n}$ be given by (ii); (iii) holds by letting $p=u_{(\alpha_{1},\ldots,\alpha_{n})}$. 

(iii) $\Rightarrow$ (i) By assumption, for any $A_1 \in u_1,\ldots, A_n \in u_n$, $(A_1 \times \ldots \times A_n) \cap R \not = \emptyset$, i.e., $\exists (a_1, \ldots  a_n) \in (A_1 \times \ldots \times A_n)$ such that $R(a_1 , \ldots , a_n)$ holds. Since $A_1 , \ldots , A_n$ are arbitrary, $R^{w} (u_1 , \ldots , u_n)$ holds by definition.

(ii) $\Rightarrow$ (iv) Without loss of generality we let $i=1$, the general case being similar. For all $A_{2}\in u_{2},\ldots, A_{n}\in u_{n}$ let $J_{A_{2},\ldots,A_{n}}:=\{x\in I_{1}\mid \exists a_{2}\in A_{2},\ldots, a_{n}\in A_{n} \ \left(x,a_{2},\ldots,a_{n}\right)\in R\}$. By (ii), $J_{A_{2},\ldots,A_{n}}\in u_{1}$. In particular, \[H_{A_{2},\ldots,A_{n}}(\alpha_{1}):=\{(\alpha_{2},\ldots,\alpha_{n})\in \,^{\ast}\left(A_{2} \times \ldots \times A_{n}\right)\mid (\alpha_{1},\alpha_{2}\ldots,\alpha_{n})\in \,^{\ast}R\}\neq\emptyset.\] We conclude by taking $(\alpha_{2},\ldots,\alpha_{n})\in\bigcap_{A_{2}\in u_{2},\ldots,A_{n}\in u_{n}}H_{A_{2},\ldots,A_{n}}(\alpha_{1})$, which is nonempty by saturation.

(iv) $\Rightarrow$ (ii) This is immediate.\end{proof}

The analogue of Theorem \ref{canonical-relation-extension-generator-theorem-01} for the strong extension is the following result, whose proof we omit as the equivalence of (i) and (iii) follows directly from the definition of strong extension (Definition \ref{relation-strong-extension}) and the other two equivalences follow from Remark \ref{remark:quantifier-switch}.

\begin{theorem}\label{strong-relation-extension-generator-theorem-01}
Let $I_{1},\ldots,I_{n}$ be nonempty sets, and let $R\subseteq I_1 \times \ldots \times I_n$. For all $u_1 \in \beta I_1 , \ldots , u_n \in \beta I_n$ the following are equivalent.
\begin{enumerate}[label=(\roman*)]
        	\item $R^s (u_1 , \ldots , u_n)$;
            \item $\exists (\alpha_1 , \ldots , \alpha_n) \in \mu(u_1 \otimes \ldots \otimes u_n) \ {^*R} (\alpha_1 , \ldots , \alpha_n)$;
            \item $R \in u_{1}\otimes\ldots\otimes u_n$;
	
    \item $\forall (\alpha_1 , \ldots , \alpha_n) \in \mu(u_1 \otimes \ldots \otimes u_n) \ {^*R} (\alpha_1 , \ldots , \alpha_n)$.
\end{enumerate}
\end{theorem}

\begin{remark} Conditions (iii) in Theorem \ref{relation-canonical-extension} and \ref{canonical-relation-extension-generator-theorem-01} are the reason why we decided to use the names ``weak'' and ``strong'': in both cases we are checking if $R$ belongs to an ultrafilter extending $u_{1}\times\dots\times u_{n}$, but in the weak case this property being witnessed by any such ultrafilter suffices, whilst in the strong case we fix exactly which ultrafilter to check, the tensor product $u_{1}\otimes\dots\otimes u_{n}$.\end{remark}

\subsection{Weak and strong extensions of equivalence and pre-order relations}\label{preservation-of-relation-properties-section}

Theorem \ref{canonical-relation-extension-generator-theorem-01} and Theorem \ref{strong-relation-extension-generator-theorem-01} allow to easily prove basic properties of weak and strong extensions of relations. Related results, with a much stronger focus on topological aspects, were obtained by Saveliev in \cite{SavelievTwoTypes}. Here we follow an approach closer to that of {\v{S}}obot in \cite{vsobot2015divisibility}; in fact, in this section we generalize his results on extensions of equivalence relations on $\N$ to equivalence relations defined on arbitrary sets, and extend them further to pre-orders.

First, it is relatively immediate to prove that weak extensions preserve reflexivity, symmetry, transitivity and connectedness\footnote{We recall that $R$ is connected when for all $x\neq y\in I$ $R(x,y)\vee R(y,x)$ holds.}.
\begin{theorem}\label{widetilde-R-equivalence-relation-theorem}
Let $I$ be a set and $R \subseteq I^2$ be a relation. The following properties hold:
\begin{enumerate}[label=(\roman*)]
    \item if $R$ is reflexive then $R^{w}$ is reflexive;
    \item if $R$ is symmetric then $R^{w}$ is symmetric;
    \item if $R$ is transitive then $R^{w}$ is transitive;
    \item if $R$ is connected then $R^{w}$ is connected.
\end{enumerate}

In particular, if $R$ is an equivalence relation, then $R^{w}$ is an equivalence relation, and if $R$ is a total pre-order then $R^{w}$ is a total pre-order.
\end{theorem}
\begin{proof} First, let us observe that if $R$ is reflexive/symmetric/transitive/total then $^{\ast}R$ is the same, by transfer. 

(i) Let $u \in \beta I$, $\alpha \in \mu(u)$. By our assumption $^{\ast}R(\alpha, \alpha)$ holds, hence by Theorem \ref{canonical-relation-extension-generator-theorem-01}.(ii), $R^{w}(u,u)$ holds.

(ii) Let $u,v \in \beta I$. Assume that $R^{w}(u,v)$ holds. By Theorem \ref{canonical-relation-extension-generator-theorem-01}.(ii), there exist $\alpha\in \mu(u),\beta \in \mu(v)$ such that $^{\ast} R(\alpha,\beta)$ holds, hence $^{\ast} R(\beta,\alpha)$ holds, and we conclude again by Theorem \ref{canonical-relation-extension-generator-theorem-01}.(ii).

(iii) Let $u,v,w \in \beta I$, and assume that $R^{w}(u,v)$ and $R^{w}(v,w)$ hold. By Theorem \ref{canonical-relation-extension-generator-theorem-01}.(ii) there exist $\alpha\in\mu(u),\beta\in\mu(v)$ such that $^{\ast}R(\alpha,\beta)$. By Theorem \ref{canonical-relation-extension-generator-theorem-01}.(iv), there exists $\gamma\in\mu(w)$ such that $^{\ast}R(\beta,\gamma)$ holds. By transitivity of $^{\ast}R$ hence $^{\ast}R(\alpha,\gamma)$ holds,  and we conclude again by Theorem \ref{canonical-relation-extension-generator-theorem-01}.(ii). 

(iv) Let $u,v \in \beta I$ with $u \not = v$. Pick $\alpha\in\mu(u),\beta\in\mu(v)$. Then either $^{\ast}R(\alpha,\beta)$ or $^{\ast}R(\beta,\alpha)$ holds. By Theorem \ref{canonical-relation-extension-generator-theorem-01}.(ii) we conclude, as the first case entails that $R^{w}(u,v)$ holds, the latter entails that $R^{w}(v,u)$ holds.\end{proof}

As already known from the works of {\v{S}}obot, the analogue of Theorem \ref{widetilde-R-equivalence-relation-theorem} for strong extensions fails in $\N$, as reflexivity and symmetry are not always preserved. This, again, is a general fact: by Theorem \ref{strong-relation-extension-generator-theorem-01}.(iv), reflexivity of $R^{s}$ holds if and only if, for all tensor pairs $(\alpha,\beta)$, $^{\ast}R(\alpha,\beta)$ holds, whilst symmetry of $R^{s}$ holds if and only if, for all tensor pairs $(\alpha,\beta)$, $R^{s}(\alpha,\beta)$ implies $R^{s}(\gamma,\delta)$ whenever $(\gamma,\delta)\sim (\beta,\alpha)$. These properties  can fail independently, as shown in the following examples.

\begin{example}\label{R-s-non-equivalence-relation-example}
Let $R$ be the equality on $\N^{2}$. Then $R^{s}$ is not reflexive, as by Proposition \ref{proprietà tensori}.\ref{T2} it is easily seen that whenever $u$ is non principal, if $(\alpha,\beta)\in\mu(u\otimes u)$ then $\alpha<\beta$. On the other hand, $R^{s}$ is symmetric (and transitive), as $R^{s}(u,v)$ entails $u=v$ and $u$ principal, hence $R^{s}(v,u)$ holds as well. Notice that, maybe surprisingly, the relation $\neq^{s}$ is reflexive on non principal ultrafilters. 
\end{example}

\begin{example}\label{R-s-non-symmetry-example}

Let $R =\{(2n,2m)\mid n,m\in\N\} \cup \{(2n+1,2m+1)\mid n,m\in\N\}\cup \{ (2n , 2m+1) \mid n \leq m \} \cup \{ (2m+1 , 2n) \mid n \leq m \} \subseteq \mathbb{N}^2$
which is, by definition, reflexive and symmetric (but not transitive). Clearly, $R^s$ is reflexive, since for all $(\alpha, \beta) \in \mu(u \otimes u)$, $\alpha$ and $\beta$ have the same parity, hence $^{\ast}R (\alpha, \beta)$ holds, and we conclude by Theorem \ref{strong-relation-extension-generator-theorem-01}.(iv). As for the symmetry of $R^{s}$, let $u,v$ be nonprincipal ultrafilters such that $u$ contains the set of even numbers, $v$ contains the set of odd numbers. By Proposition \ref{proprietà tensori}.\ref{T2}, if $(\alpha,\beta)\in \mu(u \otimes v)$ then necessarily $\alpha$ is even, $\beta$ is odd and $\alpha < \beta$. By Theorem \ref{strong-relation-extension-generator-theorem-01}.(iv), $R^{s}(u,v)$ holds. However, if $(\gamma,\delta)\in \mu(v\otimes u)$ then $\gamma$ is odd, $\delta$ is even and $\gamma<\delta$ hence, again by Theorem \ref{strong-relation-extension-generator-theorem-01}.(ii), $R^{s}(v,u)$ does not hold. This shows that $R^{s}$ is not symmetric.
\end{example}

Connectedness is another property that is not necessarily preserved by strong extensions, as proven by the following example.

\begin{example}
Consider the relation $> \subseteq \mathbb{N}^2$ and observe that it is connected. Now, consider two non-principal ultrafilters $u,v \in \beta \mathbb{N}$ and observe that, for $(\alpha,\beta) \in \mu(u \otimes v)$, we have $\beta > \alpha$ by Proposition \ref{proprietà tensori}.\ref{T2}, hence $u >^s v$ does not hold and, for the same reasons, $v >^s u$ does not hold. This shows that $>^s$ is not connected.
\end{example}

In contrast, transitivity of $R$ always lifts to transitivity of $R^{s}$.

\begin{proposition}\label{strong-preservation-transitivity}
If $R \subseteq I^2$ is a transitive relation, then also $R^s \subseteq (\beta I)^2$ is transitive.
\end{proposition}
\begin{proof}
Let $u,v,p \in \beta I$ be such that both $R^s(u,v)$ and $R^s(v,p)$ hold. 

If $u$ is the principal ultrafilter on $i\in I$, pick $(\alpha,\beta)\in\mu (v\otimes p)$, observe that $(i,\alpha)$ and $(i,\beta)$ are tensors by Proposition \ref{proprietà tensori}.\ref{T4} and conclude as, by the transitivity of $^{\ast}R$, from $^{\ast}R(i,\alpha)$ and $^{\ast}R(\alpha,\beta)$ we deduce $^{\ast}R(i,\beta)$.

If $v$ is the principal ultrafilter on $i\in I$, pick $(\alpha,\beta)\in\mu (u\otimes p)$, observe that $(\alpha,i)$ and $(i,\beta)$ are tensors by Proposition \ref{proprietà tensori}.\ref{T4} and argue as above.

If $p$ is the principal ultrafilter on $i\in I$, pick $(\alpha,\beta)\in\mu(u\otimes v)$, notice that $(\alpha,i),(\beta,i)$ are tensors by Proposition \ref{proprietà tensori}.\ref{T4}  and argue as above.

If $u,v,p$ are nonprincipal, by Theorem \ref{strong-relation-extension-generator-theorem-01}.(ii), there exists $(\alpha,\beta)\in\mu(u \otimes v)$ such that $^{\ast}R(\alpha,\beta)$ holds. By Proposition \ref{proprietà tensori}.\ref{T8} , there exists $\gamma\in\,^{\ast}I$ such that $(\beta,\gamma)\in\mu(v\otimes p)$, and by Theorem \ref{strong-relation-extension-generator-theorem-01}.(iv) $^{\ast}R(\beta,\gamma)$ holds. By transitivity of $^{\ast}R$ hence $^{\ast}R(\alpha,\gamma)$ holds, and we conclude by Theorem \ref{strong-relation-extension-generator-theorem-01}.(ii) as $(\alpha,\gamma)\in\mu(u\otimes p)$ by Proposition \ref{proprietà tensori}.\ref{T7} .\end{proof}

Whilst Examples \ref{R-s-non-equivalence-relation-example}, \ref{R-s-non-symmetry-example} showed the independence of the preservation of reflexivity and symmetry for $R^{s}$ when $R$ is not transitive, in the transitive case the situation is simpler.

\begin{theorem}\label{strong-preservation-symmetry}
Let $R \subseteq I^2$ be a symmetric and transitive relation. The following are equivalent:
\begin{enumerate}[label=(\roman*)]
    \item $R^{s}$ is reflexive; 
    \item $R^{s}$ is symmetric and total\footnote{A binary relation $S$ on $X\times Y$ is total when for all $x\in X$ there exists $y\in Y$ such that $S(x,y)$ holds.};
    \item $R^{s}$ is total;
    \item $R^{s}$ is an equivalence relation. 
    \end{enumerate}
\end{theorem}

\begin{proof} (i)$\Rightarrow$ (ii) The second part of the claim is trivial, as $R^{s}(u,u)$ holds for all $u$ by assumption. To prove that $R^{s}$ is symmetric, let $u,v \in \beta I$ be such that $R^s(u,v)$. By Theorem \ref{strong-relation-extension-generator-theorem-01}.(ii), $^{\ast}R(\alpha,\beta)$ holds for some $(\alpha,\beta)\in\mu(u\otimes v)$. As $R$ is symmetric, $^{\ast}R(\beta,\alpha)$ holds as well. By Proposition \ref{proprietà tensori}.\ref{T1}, \ref{T8} there exists $\beta^{\prime}\in\mu(v)$ such that $(\beta^{\prime},\beta,\alpha)\in\mu (v\otimes v\otimes u)$. By Proposition \ref{proprietà tensori}.\ref{T6} hence $(\beta^{\prime},\beta)\in \mu(v\otimes v), (\beta^{\prime},\alpha)\in \mu(v\otimes u)$. By reflexivity of $R^{s}$ and Theorem \ref{strong-relation-extension-generator-theorem-01}.(iv), $^{\ast}R(\beta^{\prime},\beta)$ holds. By transitivity, hence, $^{\ast}R(\beta^{\prime},\alpha)$ holds, and we conclude by Theorem \ref{strong-relation-extension-generator-theorem-01}.(ii).

(ii)$\Rightarrow$ (iii) This is immediate.

(iii)$\Rightarrow$ (i) Let $u\in\beta I$. Let $v\in \beta I$ be such that $R^{s}(u,v)$ holds. Hence $^{\ast}R(\alpha,\beta)$ holds for all $(\alpha,\beta)\in\mu(u\otimes v)$ by Theorem \ref{strong-relation-extension-generator-theorem-01}.(iv). Fix $(\alpha,\beta)\in\mu(u\otimes v)$. Similarly as in the proof of (i)$\Rightarrow$(ii), by Proposition \ref{proprietà tensori}.\ref{T6},\ref{T8} there exists $\alpha^{\prime}\in\mu(u)$ such that $(\alpha^{\prime},\alpha)\in\mu(u\otimes u), (\alpha^{\prime},\beta)\in\mu(u\otimes v)$. $^{\ast}R(\alpha^{\prime},\beta)$ holds, $^{\ast}R(\beta,\alpha)$ holds by symmetry of $R$, hence $^{\ast}R(\alpha,\alpha^{\prime})$ holds by transitivity and $^{\ast}R(\alpha^{\prime},\alpha)$ holds by symmetry. We conclude that $R^{s}(u,u)$ holds by Theorem \ref{strong-relation-extension-generator-theorem-01}.(ii). 

(iv)$\Rightarrow$ (i) This is immediate.

(i)$\Rightarrow$ (iv) $R^{s}$ is transitive as $R$ is, by Proposition \ref{strong-preservation-transitivity}. We conclude as (i) is equivalent to (ii), which gives us reflexivity and symmetry of $R^{s}$.\end{proof}

\begin{remark} Example \ref{R-s-non-equivalence-relation-example} shows that symmetry of $R^{s}$ does not suffice for $R^{s}$ to be an equivalence relation, even when transitivity holds. \end{remark}

For equivalence relations we arrive to the following characterization.

\begin{remark}\label{equivalence-relation-preserved-by-strong-extension-characterization-theorem}
For an equivalence relation $R \subseteq I^2$, the following are equivalent:
\begin{enumerate}[label=(\roman*)]
    \item $R^s$ is an equivalence;
	\item $R^s$ is reflexive;
    \item $R^{s}$ is total.
\end{enumerate}
\end{remark}

The situation for the strong extensions of pre-orders is similar to that of equivalence relations. In fact, the following analogue of Theorem \ref{strong-preservation-symmetry} holds.

\begin{theorem}
Let $R \subseteq I^2$ be a connected pre-order. The following are equivalent:
\begin{enumerate}[label=(\roman*)]
	\item $R^s$ is reflexive;
    \item $R^{s}$ is a connected pre-order.
\end{enumerate}
\end{theorem}
\begin{proof}
(i) $\Rightarrow$ (ii) Transitivity of $R^{s}$ follows from Proposition \ref{strong-preservation-transitivity}; we are left with proving that $R^{s}$ is connected.

Let $u,v \in \beta I$ be such that $u \not = v$ and suppose that $\lnot R^s(u,v)$, i.e.,
$$\lnot ( \exists (\alpha,\beta) \in \mu(u \otimes v) \, ^{\ast}R(\alpha,\beta) );$$
hence
\begin{equation}\label{pre-order-theorem-equation-01}
\forall (\alpha,\beta) \in \mu(u \otimes v) \, \lnot \,^{\ast}R(\alpha,\beta) .
\end{equation}

If one between $u$ and $v$ is principal, the thesis follows easily from the fact that every pair with a standard component is a tensor pair. So suppose both $u$ and $v$ are non-principal. Pick $(\beta , \alpha) \in \mu(v \otimes u)$ and $\beta^\prime \in \mu(v)$ such that $(\alpha , \beta^\prime) \in \mu(u \otimes v), (\beta,\beta^{\prime})\in \mu(v\otimes v)$, which exist by Proposition \ref{proprietà tensori}.\ref{T6}, \ref{T8}. By (\ref{pre-order-theorem-equation-01}), $\neg\,^{\ast}R(\alpha,\beta^{\prime})$ holds, hence $^{\ast}R(\beta^{\prime},\alpha)$ holds as $^{\ast}R$ is connected. By reflexivity of $R^{s}$, ${^\ast} R(\beta , \beta^{\prime})$ holds and, by transitivity of ${^\ast}R$, ${^\ast}R(\beta , \alpha)$ holds. Since $(\beta , \alpha) \in \mu(v \otimes u)$, we conclude by Theorem \ref{strong-relation-extension-generator-theorem-01}.(ii).

(ii) $\Rightarrow$ (i) This is immediate.
\end{proof}

\section{Mixed extensions of relations}\label{generalization-of-self-divisible-ultrafilters-section}

In \cite{vsobot2021congruence}, \v{S}obot introduced the notion of ``congruence modulo ultrafilter'' in an attempt to extend the notion of congruence modulo integer. In his notation, fixing $p\in\beta\Z\setminus\{0\}$, $\equiv_p \subseteq (\beta \Z)^2$ is the relation such that for all $u,v \in \beta \Z$ $u \equiv_p v$ holds if and only if
$$(\exists  \gamma \in \mu(p)) \, (\exists (\alpha,\beta) \in \mu(u \otimes v)) \, (\alpha \equiv \beta\mod \gamma) .$$

This notion of congruence has limitations: it is not, in general, an equivalence relation. On the other hand, the strong extension of the congruence relation, $C^{s}(p,\cdot,\cdot)$ (for $C(z,x,y):= x\equiv y\mod z$), is always an equivalence relation. In \cite{di2023self}, it was shown that $\equiv_{p}$ coincides with $R^{s}(p,\cdot,\cdot)$ precisely when $p$ is a self-divisible ultrafilter, defined as follows.

\begin{definition*}
An ultrafilter $p\in\beta\Z\setminus\{0\}$ is self-divisible if $\{n\in\Z\mid \{m\in\Z \ \text{s.t.}\ n|m\}\in p\}\in w$.  
\end{definition*} 

In this section, we want to show that many of the results proven for the congruence case in \cite{vsobot2021congruence} and \cite{di2023self} can be generalized to relations satisfying some mild algebraic property.  We start by giving the obvious generalization of the notion of self-divisibility for binary relations (the case of $n$-ary relations would be similar).

\begin{definition}\label{def self divisible ultra}
    Let $I$ be an infinite set, and $R\subseteq I^{2}$ be a relation. We say that $u\in\beta I$ is a self-$R$ ultrafilter if $\{i\in I\mid \{j\in I\mid (i,j)\in R\}\in u\}\in u$.
\end{definition}

\begin{remark}\label{remark} Notice that, by its definition and by Theorem \ref{strong-relation-extension-generator-theorem-01}, $u$ is self-$R$ if and only if each of the following equivalent properties holds: 
\begin{enumerate}[label=(\roman*)]
    \item $\{(i,j)\in I^{2}\mid R(i,j)\}\in u\otimes u$;
    \item $R^{s}(u,u)$;
    \item $\exists(\alpha,\beta)\in\mu(u\otimes u) \ ^{\ast}R(\alpha,\beta)$;
    \item $\forall(\alpha,\beta)\in\mu(u\otimes u) \ ^{\ast}R(\alpha,\beta)$.
\end{enumerate}
\end{remark}

Notice that the equivalent characterization given in Remark \ref{remark}.(ii) highlights that being self-$R$ is related to the reflexivity of the strong extension of $R$ on $u$; when $R^{s}$ is reflexive, all ultrafilters are self-$R$.

The main result connecting self-divisible ultrafilters and extensions of the congruence $C(x,y,z)$ is \cite[Theorem 4.10]{di2023self}, whose relevant part for our inquiries\footnote{Theorem 4.10 in \cite{di2023self} actually shows the equivalence between five different conditions, two of which are peculiar to division and will not be extended to arbitrary relations here.} we recall here for the readers' convenience.

\begin{theorem*} For every $p\in\beta\Z\setminus\{0\}$, the following are equivalent.
\begin{enumerate}
    \item The ultrafilter $p$ is self-divisible.
    \item The relations $\equiv_{p}$, $C^{s}(p,\cdot,\cdot)$ coincide.
    \item The relation $\equiv_{p}$ is an equivalence relation.
\end{enumerate}    
\end{theorem*}

To generalize the theorem above, we now consider ternary relations $R \subseteq J \times I^2$. Our first goal is to understand the behavior of the binary relations $R^{s}(p,\cdot,\cdot)$ whilst $p$ varies in $\beta J$.

When $p=j\in J$ is principal, the situation is simple.

\begin{lemma}\label{strong-and-index-restriction-commute-lemma}
Let $j \in J$ and let $R\subseteq J\times I^{2}$. Then $\left(R(j,\cdot,\cdot)\right)^w = R^w(j,\cdot,\cdot)$ and $\left(R(j,\cdot,\cdot)\right)^s = R^{s}(j,\cdot,\cdot)$.
\end{lemma}
\begin{proof} We will use repeatedly Theorems \ref{canonical-relation-extension-generator-theorem-01} and \ref{strong-relation-extension-generator-theorem-01}. 

In the weak extension case, let $u,v\in\beta I$. Then $(R(j,\cdot,\cdot))^{w}(u,v)$ holds if and only if there exists $\alpha\in\mu(u),\beta\in\mu(v)$ such that $^{\ast}R(j,\alpha,\beta)$ holds if and only if, as $\mu(j)=\{j\}$, there exists $i\in \mu(j),\alpha\in\mu(u),\beta\in\mu(v)$ such that  $^{\ast}R(i,\alpha,\beta)$ holds, if and only if $R^{w}(j,u,v)$ holds.

Similarly, in the strong extension case, let $u,v \in \beta I$. $R(j,\cdot,\cdot)^s (u,v)$ holds if and only if $\exists (\alpha,\beta) \in \mu(u \otimes v) \, ^{\ast}R(j,\alpha,\beta)$, which happens if and only if $(j,\alpha,\beta) \in \mu(j \otimes u \otimes v)$ and $^{\ast}R(j,\alpha,\beta)$ holds, as $\mu(j)=\{j\}$ and, as $j\in J$, we have that $(j,\alpha,\beta)$ is tensor if and only if $(\alpha,\beta)$ is by Proposition \ref{proprietà tensori}.\ref{T4}, \ref{T6}. Finally, $^{\ast}R(j,\alpha,\beta)$ is equivalent to $R^{s}(j,u,v)$.
\end{proof}

  {\v{S}}obot showed that $C^s(p,\cdot,\cdot) \subseteq (\beta \Z)^2$ is an equivalence relation for arbitrary $p\in\beta\Z\setminus\{0\}$. This fact can be seen as a particular case of the following general result, that strengthens Theorem \ref{strong-preservation-symmetry}.

\begin{theorem}\label{strong-index-reduction-general-equivalence-relation-theorem}
Let $R \subseteq J \times I^2$. Assume that, for all $j \in J$, $R(j,\cdot,\cdot)$ is an equivalence relation whose strong extension is reflexive. Then $R^s(p,\cdot,\cdot)$ is an equivalence relation for all $p \in \beta J$.
\end{theorem}
\begin{proof} Fix $p\in\beta J$. First, we notice that if $p\in J$ is principal, the claim follows from Lemma \ref{strong-and-index-restriction-commute-lemma} and Theorem \ref{strong-preservation-symmetry}.

Suppose now that $p \in \beta J \setminus J$. We will repeatedly use the nonstandard characterization of $R^{s}$ without mention.

Reflexivity. Let $u \in \beta I$. By Proposition \ref{proprietà tensori}.\ref{T1}, it is immediate to see that $R^{s}(p,u,u)$ holds if and only if 
\[X=\{j\in J\mid R^{s}(j,u,u)\}\in p.\]

By assumption and by Lemma \ref{strong-and-index-restriction-commute-lemma}, $R^{s}(j,\cdot,\cdot)$ is reflexive for all $j \in J$, hence $X = J\in p$.

Symmetry. Let $u,v \in \beta I$ be such that $R^s(p,u,v)$ holds. Pick $(\gamma,\alpha,\beta)\in\mu(p\otimes u\otimes v)$ such that $^{\ast}R(\gamma,\alpha,\beta)$ holds. By our hypothesis and transfer, $^{\ast}R(\gamma,\cdot,\cdot)$ is symmetric, hence $^{\ast}R(\gamma,\beta,\alpha)$ holds. 

By Proposition \ref{proprietà tensori}.\ref{T6} $(\gamma,\alpha)$ is a tensor pair, hence by reflexivity of $^{\ast}R(\gamma,\cdot,\cdot)$ and by Proposition \ref{proprietà tensori}.\ref{T7}, $^{\ast}R(\gamma,\alpha,\alpha^{\prime})$ holds for all $(\alpha,\alpha^{\prime})\in\mu(u\otimes u)$. In particular, by Proposition \ref{proprietà tensori}.\ref{T8}, \ref{T6} we can take $(\alpha,\alpha^{\prime})$ with $(\beta,\alpha^{\prime})$ tensor. But then $^{\ast}R(\gamma,\beta,\alpha)$ and $^{\ast}R(\gamma,\alpha,\alpha^{\prime})$ holds and, as $^{\ast}R(\gamma,\cdot,\cdot)$ is an equivalence relation by transfer, we deduce that $^{\ast}R(\gamma,\beta,\alpha^{\prime})$ holds. We conclude as $(\gamma,\beta,\alpha^{\prime})$ is a tensor triple by Proposition \ref{proprietà tensori}.\ref{T7} and \ref{T6}. 

Transitivity. Let $u,v,t \in \beta I$ such that $R^s(p,u,v)$ and $R^s(p,v,t)$ hold. Let $(\gamma,\alpha,\beta)\in\mu(p,u,v)$ and, by Proposition \ref{proprietà tensori}.\ref{T8}, let $(\gamma,\beta,\delta)\in\mu(p\otimes v\otimes t)$. Then $^{\ast}R(\gamma,\alpha,\beta)$ and $^{\ast}R(\gamma,\beta,\delta)$ holds, hence by transitivity of $^{\ast}R(\gamma,\cdot,\cdot)$ also $^{\ast}R(\gamma,\alpha,\delta)$ holds. We conclude as $(\gamma,\alpha,\delta)\in \mu(p\otimes u\otimes t)$ by Proposition \ref{proprietà tensori}.\ref{T6}, \ref{T7}. 
\end{proof}

\begin{remark}
    Lemma 5.6 in \cite{vsobot2021congruence}, namely the fact that the strong extension of the congruence is an equivalence relation, is a particular instance of Theorem \ref{strong-index-reduction-general-equivalence-relation-theorem}.
\end{remark}

$R^{s}(p,\cdot,\cdot)$ is, under the hypotheses of Theorem \ref{strong-index-reduction-general-equivalence-relation-theorem}, an equivalence relation. However, as said at the beginning of this section, in many applications, the equivalence relations that one considers are built starting from algebraic operations, and are expected to be well-behaved with respect to such operations. This can be achieved by considering the mixed extensions of relations, which are the generalization of $\equiv_{p}$.

\begin{definition}\label{def mix ext} Let $R\subseteq J\times I^{2}$ be a relation. We let the mixed extension of $R$, denoted as $R^{m}\subseteq \beta J \times (\beta I)^{2}$, to be the relation such that, for all for $p \in \beta J$ and $u,v \in \beta I$, $R^{m} (p,u,v)$ holds if and only if
$$\exists \gamma \in \mu (p) \, \exists (\alpha,\beta) \in \mu (u \otimes v) \, ^{\ast}R(\gamma,\alpha,\beta) .$$
\end{definition}

\begin{remark} By Theorem \ref{canonical-relation-extension-generator-theorem-01}, we can have an intuition of $R^{m}$ as follows: first take $R^{w}$ on $\beta J\times (\beta I)^{2}$, then restrict it to $\beta J \times T$, where $T=\{u\in\beta(I^{2})\mid u \ \text{is a tensor product of ultrafilters in} \ \beta I\}$. 
\end{remark}

\begin{remark} Mixed extensions are actually more natural than they look. For example, given any $f:I^{2}\rightarrow J$, if one considers the relation $R_{f}(z,x,y)$ defined by $f(x,y)=z$, then it is easily seen that $R_{f}^{m}(p,u,v)$ holds precisely when $p=\overline{f}(u\otimes v)$, where we denote by $\overline{f}$ the unique continuous extension of $f$ to $\beta (I^{2})$. This can be used to formalize important notions, such as idempotency: in fact, if $\bullet$ is a binary operation on $I$, saying that $u\in\beta I$ is idempotent with respect to $\overline{\bullet}$ is equivalent to say, in the notations above, that $R_{\bullet}^{m}(u,u,u)$ holds.    
\end{remark}

The above remark, specialized to $+:\N^{2}\rightarrow \N$, gives also an example of a relation for which weak, mixed and strong extensions all behave differently. In fact, letting $R(x,y,z)$ be the relation defined by $x=y+z$, it is immediate to see, using the nonstandard characterizations, that for all $u,v,p\in\beta\N$:
\begin{itemize}
    \item $R^{w}(p,u,v)$ holds whenever $p$ extends the set $\{A+B\mid A\in u, B\in v\}$, i.e. whenever a generator of $p$ can be written as the sum of a generator of $u$ and one of $v$; for example, $R^{w}(2u , u , u)$ holds;
    \item $R^{m}(p,u,v)$ holds if and only if $p = u \oplus v$;
    \item $R^{s}(p,u,v)$ holds if and only if $p,u,v$ are actually principal and $p = u \oplus v$. In fact, as soon as at least one between $\alpha,\beta,\gamma$ is infinite, Proposition \ref{proprietà tensori}.\ref{T2} forces $\neg(\,^{\ast}R(\gamma,\alpha,\beta))$ whenever $(\gamma,\alpha,\beta)$ is a tensor triple.
\end{itemize}

Returning to the general setting, it is clear that $R^s \subseteq R^{m} \subseteq R^{w}$. As before, we are interested in the induced binary relations $R^{m}(p,\cdot,\cdot) \subseteq (\beta I)^2$ for $p \in \beta J$; in analogy with what was done in \cite{di2023self}, our main goal is to understand for which $p\in\beta J$ one has that $R^{m}(p,\cdot,\cdot)=R^{s}(p,\cdot,\cdot)$. This is the case when $p$ is principal, as we now prove.

\begin{proposition}
If $j \in J$ then $R^{m}(j,\cdot,\cdot)=R^{s}(j,\cdot,\cdot)$.
\end{proposition}

\begin{proof}
Let $u,v \in \beta I$. As $\mu(j)= \{ j \}$ and $(j,\alpha,\beta)$ is tensor if and only if $(\alpha,\beta)$ is, by definition $R^{m}(j,u,v)$ holds if and only if $\exists (\alpha,\beta) \in \mu(u \otimes v) \, ^{\ast}R(j,\alpha,\beta)$ holds, i.e. if $R^{s}(j,u,v)$ holds.
\end{proof}

As said, when $p$ is nonprincipal, in general $R^{m}(p,\cdot,\cdot)\neq R^{s}(p,\cdot,\cdot)$, see e.g. \cite[Example 3.1]{di2023self}. However, under some basic algebraic conditions on $R$, the equality can be restored for certain choices of $p$.

\begin{definition}\label{condizioni}
    Let $I$ be a set, and let $e\in I$. Let $R\subseteq  I^{3}$ be a relation.

    We say that $R$ is $e$-transitive if $\forall i,j\in I\setminus\{e\}, \forall a,b\in I$ $R(i,j,e)\wedge R(j,a,b)\Rightarrow R(i,a,b)$.

    We say that $R$ is $e$-reflexive if $\forall i\in I\setminus\{e\}$ $R(i,i,e)$ holds.
\end{definition}

\begin{remark} In the case of the usual congruence relation on $\Z^{3}$, setting $e=0$, it is easily seen that $0$-transitivity corresponds to the fact that if $a\equiv b\mod j$ and $i$ divides $j$ then $a\equiv b\mod i$, and $0$-reflexivity corresponds to the fact that, for all $a\neq 0$, $a\equiv 0\mod a$.\end{remark}

\begin{example}\label{examrelat} Definition \ref{condizioni} covers many natural examples.

\begin{enumerate}
    \item With $I=\Z$ and $e=0$, the congruence relation $R(x,y,z)\Leftrightarrow y\equiv z\mod x$ is easily seen to be $0$-transitive and $0$-reflexive.
    \item More in general, let $(I,\cdot)$ be a group and $e$ its identity. Let $R(x,y,z)$ be the relation $\exists k\in\Z \ yz^{-1}=x^{k}$, where $x^{k}:=\overset{k \ \text{times}}{\overbrace{x\cdot\ldots\cdot x}}$. This relation is $e$-transitive: assume $R(x,y,z)$ and $R(t,x,e)$. Then there exists $h,k\in\Z$ such that $yz^{-1}=x^{k}$ and $xe^{-1}=x=t^{h}$. Hence $yz^{-1}=t^{hk}$, so $R(t,y,z)$ holds. Moreover, this relation is also trivially $e$-reflexive.
    \item Let $I=\C$ and $e=1$. Consider the relation $R \subseteq \C \setminus \{ 1 \} \times \C^2$ such that\footnote{With $\Q(x)$ we denote the field generated by $x \in \C$ over $\Q$.}
$$R(x,a,b) \iff \Q(x,a) = \Q(x,b) .$$

This relation is $1$-transitive. In fact, let $x,y \in \C \setminus \{ 1 \}$ and $a,b \in \C$ be such that $R(x,y,1) \land R(y,a,b)$, i.e., $\Q(x,y) = \Q(x,1) = \Q(x)$ and $\Q(y,a) = \Q(y,b)$. Then $\Q(x,a) = \Q(x,y,a) = \Q(x,y,b) = \Q(x,b)$, which means that $R(x,a,b)$ holds. Moreover, $R$ is also $1$-reflexive, as for all $x \in \C \setminus \{ 1 \}$, $\Q(x,x) = \Q(x) = \Q(x,1)$, i.e., $R(x,x,1)$ holds.
\end{enumerate}
\end{example}

When $R$ is $e$-transitive and $e$-reflexive for some $e$, we can extend the relationship between self-divisible ultrafilters and equality between mixed and strong extensions of congruences (namely, \cite[Theorem 4.10]{di2023self}) as follows.

\begin{theorem}\label{partial-R-self-divisible-generalization-theorem-01}
Let $I$ be a set, and let $e\in I$. Let $R \subseteq I \setminus \{ e \} \times I^2$, and let $p \in \beta I \setminus \{ e \}$. Then:

\begin{enumerate}[label=(\roman*)]
        	\item if $R$ is $e$-transitive and $p$ is a self-$R(\cdot,\cdot,e)$ ultrafilter, then the relations $R^{m}(p,\cdot,\cdot)$ and $R^s(p,\cdot,\cdot)$ coincide;
            \item if the relations $R^{m}(p,\cdot,\cdot)$ and $R^{s}(p,\cdot,\cdot)$ coincide and $R$ is $e$-reflexive, then $p$ is a self-$R(\cdot,\cdot,e)$ ultrafilter.
            \end{enumerate}

In particular, when $R$ is both $e$-transitive and $e$-reflexive, then $R^{m}(p,\cdot,\cdot)=R^{s}(p,\cdot,\cdot)$ if and only if $p$ is a self-$R(\cdot,\cdot,e)$ ultrafilter.

\end{theorem}

\begin{proof} The ``in particular'' part follows immediately by (i),(ii), which we now prove.

(i) $R^{s}(p,\cdot,\cdot)\subseteq R^{m}(p,\cdot,\cdot)$ is always true. To prove the reverse inclusion, let $u,v \in \beta I$ be such that $R^{m}(p,u,v)$ holds. Let $\gamma \in \mu(p), (\alpha, \beta) \in \mu(u \otimes v)$ be such that $^{\ast}R(\gamma , \alpha , \beta) $ holds. Let $w$ be the ultrafilter on $I^{3}$ generated by $(\gamma,\alpha,\beta)$. By Proposition \ref{proprietà tensori}.\ref{T8}, we can pick $\delta$ such that $(\delta,(\gamma,\alpha,\beta))$ is a generator of $p\otimes w$.  By letting $f:I\rightarrow I$ be the identity, and $g:I^{3}\rightarrow I$ be the projection on the first coordinate, by Proposition \ref{proprietà tensori}.\ref{T9} we get that $(\delta,\gamma)$ is a tensor pair; similarly, if we let instead $g:I^{3}\rightarrow I^{2}$ be the projection on the second and third coordinate, we deduce that also $(\delta,(\alpha,\beta))$ is a tensor pair. As $p$ is a self-$R(\cdot,\cdot,e)$ ultrafilter, by Remark \ref{remark}.(iv) we have that $^{\ast}R(\delta , \gamma , e)$ holds. As $^{\ast}R(\gamma , \alpha , \beta) $ holds, by $e$-transitivity of $^{\ast}R$ we deduce $^{\ast}R(\delta , \alpha , \beta)$, which shows that $R^{s}(p,u,v)$ holds as desired.

(ii) Take any $\gamma \in \mu(p)$ and observe that, by transfer of the $e$-reflexivity condition, $^{\ast}R(\gamma,\gamma,e)$ holds. As $(\gamma,e)$ is tensor by Proposition \ref{proprietà tensori}.\ref{T4}, this proves that $R^{m}(p,p,e)$ holds. But $R^{m}(p,\cdot,\cdot)$ is equal to $R^{s}(p,\cdot,\cdot)$ by hypothesis, hence $R^{s}(p,p,e)$ holds, namely $p$ is a self-$R(\cdot,\cdot,e)$ ultrafilter.
\end{proof}

\begin{corollary} Let $I$ be a set, and let $e \in I$. Let $R \subseteq I \setminus \{ e \} \times I^2$ be an $e$-transitive relation. Assume that, for all $j \in J$, $R(j,\cdot,\cdot)$ is an equivalence relation whose strong extension is reflexive. Let $p$ be a self-$R(\cdot,\cdot,e)$ ultrafilter. Then $R^{m}(p,\cdot,\cdot)$ is an equivalence relation.\end{corollary}

\begin{proof} This is immediate from Theorem \ref{partial-R-self-divisible-generalization-theorem-01}.(i) joint with Theorem \ref{strong-index-reduction-general-equivalence-relation-theorem}.\end{proof}

\begin{example} Let us see what Theorem \ref{partial-R-self-divisible-generalization-theorem-01} tells us about the relations discussed in Example \ref{examrelat}.
\begin{enumerate}
    \item In the case of Example \ref{examrelat}.(1), self-$R(\cdot,\cdot,0)$ ultrafilters are exactly the self-divisible ultrafilters introduced in \cite{di2023self}. In this case, for $p$ self-divisible $R^{m}(p,\cdot,\cdot)=R^{s}(p,\cdot,\cdot)$ is an equivalence relation.
    
    \item In the case of Example \ref{examrelat}.(2), $p$ is a self-$R(\cdot,\cdot,e)$ ultrafilter when for all $(\alpha,\beta)\in\mu(p\otimes p)$ $\beta$ is an hyperfinite power of $\alpha$, i.e. there exists $k\in\,^{\ast}\Z$ such that $\beta=\overset{k \ \text{times}}{\overbrace{\alpha\cdot\ldots\cdot\alpha}}$. For example, when $I=\Q \setminus \{0\}, e=1$ and $\cdot$ is the usual multiplication, an ultrafilter $p\in\beta(\Q \setminus \{0\})$ is self-$R(\cdot,\cdot,1)$ if and only if for all $(\alpha,\beta)\in\mu(p\otimes p)$ $\beta$ is a power of $\alpha$; this happens for example when $p=2^{u}$ with $u$ self-divisible.
    
    As shown in Example \ref{examrelat}.(2), $R$ is both $e$-reflexive and $e$-transitive, so by Theorem \ref{partial-R-self-divisible-generalization-theorem-01} we have that $R^{s}(p,\cdot,\cdot)$ is equal to $R^{m}(p,\cdot,\cdot)$ precisely when $p$ is a self-$R(\cdot,\cdot,e)$ ultrafilter. 
     However, let us notice that, in general, there might exist $i\in I\setminus\{e\}$ such that $R^{s}(i,\cdot,\cdot)$ is not an equivalence relation. In fact, $^{\ast}R(i,\alpha,\beta)$ means that $\alpha\beta^{-1}$ belongs to the subgroup $\{i^{k}\mid k \in\,^{\ast}\Z\}$, which shows that $R^{s}(i,\cdot,\cdot)$ is reflexive only when $(I,\cdot)$ is cyclic. In the cyclic  case, we can apply Theorem \ref{strong-index-reduction-general-equivalence-relation-theorem} to deduce that $R^{m}(p,\cdot,\cdot)=R^{s}(p,\cdot,\cdot)$ is an equivalence relation when $p$ is a self-$R(\cdot,\cdot,e)$ ultrafilter.
    
    \item We now consider the relation $R$ introduced in the Example \ref{examrelat}.(3). First, let us notice that $u\in\beta\C$ is a self-$R(\cdot,\cdot,1)$ ultrafilter if and only if for all $(\alpha,\beta)\in\mu(u\otimes u)$ $^{\ast}\Q(\alpha, 1)=\,^{\ast}\Q(\alpha,\beta)$, i.e. $\beta\in\,^{\ast}\Q(\alpha)$. Using this characterization, it is simple to show that there is plenty of self-$R(\cdot,\cdot,1)$ ultrafilters. For example, let us show that this is the case when $u$ is generated by $\alpha\in\,^{\ast}\Q$. In fact, in this case, taken any $(\alpha,\beta)\in\mu(u\otimes u)$ we have that $\alpha,\beta\in\,^{\ast}\Q$. Hence $^{\ast}\Q=\,^{\ast}\Q(\alpha)=\,^{\ast}\Q(\alpha, 1)=\,^{\ast}\Q(\alpha,\beta)$, as desired.
		
		Let us show explicitly that that there are also non self-$R(\cdot, \cdot, 1)$ ultrafilters. To see this, let $N\in\,^{\ast}\N\setminus\N$ be a squarefree number. Let $\alpha=\sqrt{N}$. Then every $\beta$ with $\alpha\sim\beta$ will be the square root of a squarefree number $M\in\,^{\ast}\N\setminus\N$. In particular, if $\alpha\neq\beta$ then $\frac{M}{N}$ will not be a square, hence $\Q(\alpha,\beta)=\Q(\sqrt{N},\sqrt{M})\neq \Q(\sqrt{N})= \Q(\alpha)$, and $u_{\alpha}$ is not a self-$R(\cdot,\cdot,1)$ ultrafilter.

	Finally: as already shown, $R$ is $1$-transitive and $1$-reflexive, so we can apply Theorem \ref{partial-R-self-divisible-generalization-theorem-01} to deduce that $R^{s}(p,\cdot,\cdot)$ is equal to $R^{m}(p,\cdot,\cdot)$ precisely when $p$ is a self-$R(\cdot,\cdot,1)$ ultrafilter. \end{enumerate}
\end{example}

\section{Open questions}\label{sec open questions}

Theorem 4.10 in \cite{di2023self}, in our language, shows the equivalence between three conditions: $p$ being self-divisible, $C^{s}(p,\cdot,\cdot)$ being equal to $C^{m}(p,\cdot,\cdot)$ and $C^{m}(p,\cdot,\cdot)$ being an equivalence relation. The first two conditions are equivalent, under the assumptions of $e$-transitivity and $e$-reflexivity, also in the cases under study here, as shown in Theorem \ref{partial-R-self-divisible-generalization-theorem-01}. Our first question is if the equivalence between these two conditions and $R^{m}(p,\cdot,\cdot)$ being an equivalence relation can be recovered, under the same assumptions. 

\begin{question}
Assume $R$ to be $e$-transitive and $e$-reflexive. Is it true that $R^{m}(p , \cdot , \cdot)$ is an equivalence relation if and only if $p$ is self-$R(\cdot,\cdot,e)$ (i.e. if and only of $R^{m}(p,\cdot,\cdot)=R^{s}(p,\cdot,\cdot)$)?
\end{question}

The hypotheses of $e$-transitivity and $e$-reflexivity are quite natural, but so far it is not clear if they are necessary to prove Theorem \ref{partial-R-self-divisible-generalization-theorem-01}. Our second question is, hence, the following:

\begin{question} Are there any other natural substitutes for $e$-reflexivity and $e$-transitivity that would allow to prove an analogue of Theorem \ref{partial-R-self-divisible-generalization-theorem-01}? Can these assumptions be dropped altogether? \end{question}

The research so far has focused on understanding when strong and mixed extensions are equal. It might be worth investigating also the following related problem.

\begin{question}
    Under which hypotheses does the equality $R^{m}(p,\cdot,\cdot)=R^{w}(p,\cdot,\cdot)$ hold?
\end{question}

Finally, as mentioned before, self-divisible ultrafilters seem to have a role to play in understanding combinatorial properties of the division. Our final question points in a similar direction.

\begin{question}
    What are the relations $R$ for which self-$R$ ultrafilters play a natural combinatorial role?
\end{question}

\bibliographystyle{plain}
\bibliography{references}
\end{document}